\documentclass{icmart}

\usepackage[non-sorted-cites]{amsrefs}
\usepackage{euscript}
\usepackage{amsmath}
\usepackage{amscd}
\usepackage{amssymb}
\usepackage{enumerate}
\usepackage{stmaryrd}
 \usepackage{amsfonts}
\usepackage[all]{xy}
\usepackage{eufrak}
\usepackage{textcomp}
\usepackage{tikz}
\usetikzlibrary{calc,intersections,through,backgrounds}
\newlength{\lw}
\newcommand{\bC}{{\mathbb C}}

\newcommand{\bR}{{\mathbb R}}
\newcommand{\bZ}{{\mathbb Z}}

\newcommand{\bF}{{\mathbb F}}

\newcommand{\bk}{{\mathbf k}}

\newcommand{\cE}{\mathcal E}
\newcommand{\cF}{\mathcal F}

\newcommand{\cA}{\mathcal A}

\newcommand{\cL}{\mathcal L}

\newcommand{\scrF}{\EuScript F}

\newcommand{\scrM}{\EuScript M}
\newcommand{\scrO}{\EuScript O}
\newcommand{\scrU}{\EuScript U}

\newcommand{\TQ}{T^{*}Q}
\newcommand{\Q}{Q}

\newcommand{\triv}{\tau}

\newcommand{\scrJ}{{\EuScript J}}

\newcommand{\id}{\operatorname{id}}

\newcommand{\Fuk}{\operatorname{Fuk}}
\newcommand{\Coh}{\operatorname{Coh}}

\newcommand{\Hom}{\operatorname{Hom}}

\newcommand{\Strip}{B}

\newcommand{\Pin}{\operatorname{Pin}}

\newcommand{\val}{\operatorname{val}}

\newcommand{\Vect}{\operatorname{Vect}}
\newcommand{\Aff}{\operatorname{Aff}}
\newcommand{\Diff}{\operatorname{Diff}}
\newcommand{\Id}{\operatorname{Id}}
\def\co{\colon\thinspace}

\numberwithin{equation}{section}

\newtheorem{thm}{Theorem}[section]
\newtheorem{cor}[thm]{Corollary}
\newtheorem{lem}[thm]{Lemma}
\newtheorem{prop}[thm]{Proposition}
\newtheorem{defin}[thm]{Definition}
\newtheorem{def-lem}[thm]{Definition-Lemma}
\newtheorem{conj}[thm]{Conjecture}

\theoremstyle{remark}

\newtheorem{rem}[thm]{Remark}

\newcommand{\comment}[1]{}

\contact[abouzaid@math.columbia.edu]{Department of Mathematics, Columbia University, 2990 Broadway
New York, NY 10027, USA}

\title[Family Floer cohomology and mirror symmetry]{Family Floer cohomology and mirror \\symmetry}

\author[M.~Abouzaid]{Mohammed Abouzaid} \date{\today}
\begin{document}

\begin{abstract}
Ideas of Fukaya and Kontsevich-Soibelman suggest that one can use Strominger-Yau-Zaslow's geometric approach to mirror symmetry as a torus duality to construct the mirror of a symplectic manifold equipped with a Lagrangian torus fibration as a moduli space of those simple objects of the Fukaya category which are supported on the fibres. In the absence of singular fibres, the construction of the mirror is explained in this framework, and, given a Lagrangian submanifold, a (twisted) coherent sheaf on the mirror is constructed.
\end{abstract}

\begin{classification}
Primary 53D40; Secondary 14J33.
\end{classification}
\begin{keywords}
Lagrangian Floer cohomology, Homological Mirror symmetry.
\end{keywords}
\maketitle
\section{Introduction}
\label{sec:introduction}

\subsection{Overview}
\label{sec:fluff}

Mirror symmetry is a prediction from string theory identifying invariants associated to the complex geometry of a family of Calabi-Yau manifolds  with  invariants associated to the K\"ahler geometry of a possibly different Calabi-Yau manifold that is called the mirror. In our setting, we take as definition of a Calabi-Yau manifold a complex manifold equipped with a nowhere vanishing holomorphic volume form. The original focus in mathematics was on the dualities of Hodge diamonds which gave a straightforward though non-trivial check, and on the enumerative predictions for the number of curves of a given genus and degree that went beyond the computations which could be performed using rigourous methods. While it is not reasonable to expect the existence of a mirror partner to every Calabi-Yau manifold, there are large classes of examples (e.g. toric complete intersections  \cite{batyrev,givental})  for which various original forms of the conjecture have been verified.

In \cite{Kontsevich-ICM}, Kontsevich introduced a homological version of the conjecture: the invariants to be related would be the derived category of coherent sheaves on the complex side and the Fukaya category of Lagrangian submanifolds on the symplectic side. Strominger, Yau, and Zaslow \cite{SYZ} later introduced a geometric version of the conjecture: mirror pairs should arise as dual torus fibrations over the same base; these are often called SYZ fibrations. The degenerating family can then be understood as arising from rescaling the fibres. 

It is easier to state precise versions of the SYZ conjecture (which still hold in a large class of examples) if one analyses the two sides of mirror symmetry separately, i.e. fixing a K\"ahler form on a Calabi-Yau manifold $X$ whose symplectic topology will be related to the complex geometry of a Calabi-Yau variety $Y$ over the ring of power series $\bC[[T]]$ or the analogous rings appearing naturally in symplectic topology in which real exponents are allowed; weakening one side to a formal family is related to a convergence problem in Floer theory.  

In this context, the starting point of a reformulation of the SYZ conjecture is the existence of a Lagrangian fibration $
\pi \co X \to \Q  
$ with singularities (there is still no good approach to the class of allowable singularities). The space $Y$ should then be constructed from the Fukaya category $\Fuk(X)$ as a moduli space of objects  supported on fibres of this map \cite{fukaya-K3}. While the moduli space of such objects can be described locally over the base as the dual torus fibration, the fact that we consider them as objects of the Fukaya category introduces non-trivial identifications of the local charts given by the sum of a classical term with \emph{instanton corrections} that arise from the moduli space of holomorphic discs bounded by such fibres. These corrections are expected to be expressible in terms of the geometry of the base via \emph{wall-crossing formulae} \cite{KS,gross-siebert}.

In this setting, the homological mirror conjecture asserts the existence of a derived equivalence
$
 D   \Fuk(X) \cong D \Coh(Y),
$
where both categories are linear over the appropriate version of power series rings. Much effort has gone in verifying this conjecture in certain examples \cites{seidel-quartic,sheridan,ASmith}, and extracting some of the classical statements of mirror symmetry from it \cite{KKP,costello}.  

Unfortunately, all the current proofs of mirror symmetry rely on \emph{ad hoc} methods to construct the mirror functor, neglecting the construction of the mirror as a moduli space of objects of the Fukaya category.  To address this issue, Fukaya has introduced family Floer cohomology \cite{fukaya-familyHF,fukaya-K3,Fukaya-cyclic,Fukaya-announce} as a strategy for directly assigning a sheaf on $Y$ to a Lagrangian $L$ in $X$; as noted in \cite{fukaya-K3}, the main difficulty arises from the caustics of $L$, i.e. the singularities of its  projection  to the base. In Section \ref{sec:from-lagr-grad}, we use the invariance properties of Floer cohomology under continuation maps to bypass the difficulties arising from caustics (as in  \cite[Section 6]{fukaya-familyHF}), and prove convergence in the rigid analytic sense.

\vspace{.1in}

\subsection{A twisted example}
\label{sec:example}

Since the problem is sufficiently complicated in the absence of singular fibres,  we shall henceforth only consider symplectic manifolds that admit smooth Lagrangian torus fibrations.  This class includes, for example, a codimension $\frac{n(n-1)}{2}$ subspace of the $n(2n-1)$-dimensional space of linear symplectic structures on $\bR^{2n}/\bZ^{2n}$, corresponding to those structures that can be represented as the quotient of  $\bR^{2n}$ by a lattice which intersects some Lagrangian plane in a rank $n$ subgroup. It also includes the following twisted example due to Thurston \cite{thurston}:

Equip $\bR^{4}$ with coordinates $(x_1, x_2, x_3, x_4)$, and symplectic structure
\begin{equation}
  dx_1 \wedge dx_2 + dx_3 \wedge dx_4.
\end{equation}
This form is invariant under the  transformation
  \begin{equation}
    (x_1, x_2, x_3, x_4) \to (x_1 +1, x_2, x_3, x_4+x_3)
  \end{equation}
as well as under translation by integral vectors of the form $(0,  x_2, x_3, x_4)$. The Thurston manifold is the quotient of $\bR^{4}$ by  the group generated by these transformations. 

Thurston considered the symplectic fibration obtained by forgetting the $(x_3,x_4)$ coordinates, which gives a description of this space as a twisted (flat) torus bundle over the torus.  In joint work with Auroux, Katzarkov, and Orlov \cite{AAKO} we noticed the existence of two (inequivalent) Lagrangian fibrations over this space:

(1) The fibration obtained by forgetting the $x_2$ and $x_4$ coordinates is a principal bundle on which the $(x_2,x_4)$-torus acts.  Since the total space is not trivial, this fibration admits no continuous section, and hence, in particular, no Lagrangian section.

(2) The fibration obtained by forgetting the $x_1$ and $x_4$ coordinates admits a Lagrangian section $(0,x_2,x_3,0)$.

The above two examples show that one can have Lagrangian fibrations with completely different behaviour on the same symplectic manifold. The first fibration is mirror to a gerbe on an abelian variety, while the second is mirror to a Kodaira surface. To see this, write the first fibration as the quotient of
$
  [0,1]^{2} \times \left( \bR/\bZ \right)^{2}  
$
by the equivalence relation
\begin{align}
  (0,x_3, x_2, x_4) & \sim (1,x_3, x_2, x_4 + x_3) \\
(x_1, 0, x_2, x_4) & \sim (x_1, 1, x_2, x_4). 
\end{align}
Note that the gluing maps act trivially on the homology of the fibre, so the action on the space of local systems is trivial. Since the fibres bound no holomorphic disc, the mirror is the moduli space of such local systems; it agrees with the mirror of the symplectic manifold obtained via the trivial identifications. This symplectic manifold is simply the product of two tori of area $2 \pi$, hence the mirror is a product of two (families of) elliptic curves as is well-understood by now \cite{ASmith}.

At this stage, it is clear that additional data are needed to implement mirror symmetry from this point of view.  At the most basic level, the Lagrangian fibres are null-homologous, which implies that Floer cohomology has vanishing Euler characteristic whenever one of the two inputs is a fibre.  Under mirror symmetry the fibres map to skyscraper sheaves of point, and there are therefore many coherent sheaves (e.g. vector bundles of non-vanishing rank) whose mirrors would be expected to have Floer cohomology groups with a fibre whose Euler characteristic does not vanish.

The additional data arise naturally on both sides: by obstruction theory, the failure of this torus fibration to be trivial is detected by a second cohomology group of the base  as one  can easily construct a Lagrangian section in the complement of a point.  This obstruction class is constructed in Section \ref{sec:lagr-torus-fibr} using a \v{C}ech cover, and a simple exponentiation procedure in Section \ref{sec:rigid-analyt-gerb} then yields an $\scrO^{*}$-valued second cohomology class on the mirror space; i.e. a gerbe. The correct statement of mirror symmetry involves sheaves twisted by this gerbe.

For completeness, we elaborate on the description of the mirror of the second fibration: a convenient starting point is the vector bundle $ \left( \bR/\bZ \right)^{2}   \times \bR^{2}$  over the torus in the $(x_2,x_3)$ coordinates whose fibre is the plane spanned by $(x_1,x_4)$. The key observation is that the Thurston manifold is obtained by taking the quotient of the fibre over $(x_2,x_3)$ by the lattice spanned by
$
  (1,x_3)$  and $(0,1)$.
The corresponding family of lattices in $\bR^{2}$ has non-trivial monodromy (given by an elementary transvection) around a loop in the $x_3$ direction. Constructing the mirror (complex) manifold by dualising the fibres, we conclude that the underlying smooth manifold is also a torus fibration over the torus with total Betti number $3$, hence a primary Kodaira surface \cite[p. 787-788]{Kodaira}. With a bit more care, one can avoid appealing to the classification of surfaces and identify the (complex) mirror as the degenerating family of quotients of $\bC^{*} \times \bC^{*}$ by the groups of automorphisms
\begin{align}
  (z_1, z_2) & \mapsto (z_1,  T \cdot z_2) \\
(z_1, z_2) & \mapsto (T \cdot z_1,  z_1 z_2)
\end{align}
parametrised by a variable $T$. To obtain a precise statement of Homological mirror symmetry in this setting, one then interprets the above as giving  a rigid analytic primary Kodaira surface \cite[p. 788]{ueno}.

One interesting outcome of this observation is that the two homological mirror symmetry statements imply the existence of a derived equivalence between twisted coherent sheaves on the mirror abelian variety and coherent sheaves on the mirror primary Kodaira surface. Such a result can be proved independently of mirror symmetry by exhibiting a Fourier-Mukai kernel \cite{AAKO}.

\subsection{Statement of the results}
\label{sec:statement-results}
Let $(X,\omega)$ be a compact symplectic manifold, equipped with a fibration $
\pi \co X \to \Q $
over a smooth manifold $\Q$ with Lagrangian fibres; the triple $(X,\omega,\pi)$ is called a \emph{Lagrangian fibration}; write $F_{q}$ for the fibre at a point $q$, and assume that $\pi_{2}(\Q) = 0.$

\begin{rem}

The vanishing of the second homotopy group of $\Q$ implies that $\pi_{2}(X,F_{q})$ vanishes, and hence that $F_{q}$ bounds no holomorphic disc, which implies that there are no instanton corrections, i.e. that the mirror should be the space of rank-$1$ unitary local systems on the fibres. There seem to be no known examples where this condition fails (though this can be arranged if $X$ is not assumed to be compact, or if singular fibres are allowed).

\end{rem}
The symplectic topology of such fibrations is reviewed in Section \ref{sec:lagr-torus-fibr}, but for now, recall that the tangent space of $\Q$ at a point $q$ is naturally isomorphic to $H^{1}(F_{q}, \bR) \equiv H^{1}(F_{q}; \bZ) \otimes \bR$. Arnol'd's Liouville theorem implies that the corresponding lattice in $T\Q$ arises from an integral affine structure on $\Q$. In particular, $\Q$ can be obtained by gluing polytopes in $\bR^{n}$ by transformations whose differentials lie in $SL(n,\bZ)$.

To this integral affine structure, one associates a rigid analytic space (in the sense of Tate), which we denote $Y$ (this is the same construction used in \cite{KS,FOOO-toric,tu2}): fix a field $\bk$, and consider the Novikov field
\begin{equation}
\Lambda = {\Big{\{}} \sum_{i=1}^{\infty} a_{i} t^{\lambda_i} | a_{i} \in \bk, \, \lambda_i \in \bR , \, \lambda_i \to +\infty{\Big \}}.
\end{equation}
This is a non-archimedian field with valuation $ \val \co \Lambda - \{ 0 \}  \to \bR$ 
\begin{equation}
\sum_{i=1}^{\infty} a_{i} t^{\lambda_i}  \mapsto \min( \lambda_i | a_i \neq 0 ).
\end{equation}
Denote by $U_{\Lambda}$ the units of $\Lambda$, i.e. those elements with $0$-valuation. 

As a set, the space $Y$ is the union $
  \coprod_{q \in \Q} H^{1}(F_{q}; U_{\Lambda})$. This description makes explicit the fact that $Y$ parametrises Lagrangian fibres together with the datum of a \emph{rank-$1$ $\Lambda$-local system} with monodromy in $U_{\Lambda}$. The analytic structure on $Y$ arises from the natural identifications of the first cohomology groups of nearby fibres; the explicit construction appears in Section \ref{sec:constr-mirr-space}.

Given a Lagrangian $L \subset X$, satisfying the technical conditions required to make Floer cohomology well-defined, one obtains Floer cohomology groups
$ HF^{k}( (F_q,b), L; \Lambda)$ for each  $q \in \Q$,  and  $b \in H^{1}(F_{q}; U_{\Lambda})$.

These groups are locally the fibres of coherent sheaves on $Y$, as can be seen by adapting an argument of Fukaya \cite{Fukaya-cyclic} who studied the case of self-Floer cohomology.   In order to encode the full data of the global Floer theory of $X$ on the $Y$-side, an \emph{analytic gerbe} $\alpha_{X}$ on $Y$ is introduced in Section \ref{sec:rigid-analyt-gerb}. 
\begin{thm} \label{thm:main}
There is an $\alpha_X$-twisted coherent sheaf on $Y$ whose dual fibre at the point of $Y$ corresponding to a pair $(q,b)$ is $HF^{k}((F_q,b), L; \Lambda) $.
\end{thm}
\begin{rem}
 One can interpret this theorem as an attempt to make rigourous the strategy introduced in Fukaya's announcement \cite{fukaya-familyHF}, which assigns to a Lagrangian a complex analytic sheaf \emph{assuming convergence}. In the special case of Lagrangian surfaces constructed by Hyperk\"ahler rotation, analytic continuation may provide an alternative approach for bypassing the problem arising from caustics, as noted in \cite{fukaya-K3}. As part of an ongoing project to study mirror symmetry from the point of view of family Floer cohomology \cite{tu,tu2}, Junwu Tu has an independent argument to prove a similar result.
\end{rem}

\subsection{Difficulties lying ahead}
While the above result points in the right direction, it is unfortunately not adequate for serious applications. The last section of this paper outlines the construction of an object $\cL$ in a category of $\alpha_X$-\emph{twisted sheaves of perfect complexes} that is a differential graded enhancement of the derived category of $\alpha_X$-twisted coherent sheaves on $Y$. The twisted sheaf in Theorem \ref{thm:main} is obtained from $\cL$ by passing to cohomology. 

The following conjecture makes clear why $\cL$, rather than its cohomology sheaves, is the right object to study:
\begin{conj}
If $L_1$ and $L_2$ are Lagrangians in $X$ with corresponding twisted sheaves of perfect complexes $\cL_1$ and $\cL_2$, there is an isomorphism
  \begin{equation}
    HF^{*}(L_1,L_2) \cong H^*(\Hom_{*}( \cL_1, \cL_2)) .
  \end{equation}
\end{conj}

There are two essential difficulties that arise in applying these techniques to the SYZ fibrations $X \to \Q$ that are expected to exist, for example, on Calabi-Yau hypersurfaces in toric varieties:

(1) In general, some smooth fibres may bound non-constant holomorphic discs. Assuming such fibres are \emph{unobstructed}, which essentially means that the counts of holomorphic discs with boundary on a single fibre algebraically vanish, Tu constructed the candidate (open subset) of the mirror in \cite{tu}.  The basic idea, following Fukaya \cite{fukaya-K3} and Kontsevich and Soibelman \cite{KS}, is that we should obtain the mirror space by gluing affinoids as in the uncorrected case, but the gluing maps take into account moduli spaces of holomorphic discs. One can interpret part of the program of Gross and Siebert \cite{gross-siebert} as providing such a construction for fibrations arising from toric degenerations, though it is not yet known how to prove that their construction agrees with the one intrinsic to symplectic topology. Given the appropriate technical tools (i.e. a theory of virtual fundamental chains on moduli spaces of holomorphic discs, as in \cite{FOOO}), the extension of our results to this setting should be straightforward.

(2) The construction of the mirror space from an SYZ perspective also requires understanding the Floer theory of singular fibres. Whenever the fibre is \emph{immersed}, this Floer theory is well-understood \cite{AJ}. In the simplest situation, such a fibre is an immersed Lagrangian $2$-sphere with a single double point in a $4$-manifold, and the nearby Lagrangian (torus) fibres are obtained by Lagrangian surgery \cite{polterovich}. Using the relation between moduli spaces of holomorphic discs before and after surgery \cite{FOOO}, Fukaya \cite{Fukaya-announce} has announced the construction of the mirror space in this setting. In particular, for a symplectic structure on a  $K3$-surface admitting a Lagrangian torus fibration, Fukaya's method provides a construction of the mirror space using family Floer cohomology. Unfortunately, in higher dimensions, SYZ fibrations are expected to have singular fibres which cannot be described as Lagrangian immersions, putting them beyond the reach of current techniques.

\subsection*{Acknowledgments}
It is my pleasure to thank my collaborators, on past and ongoing projects, for teaching me much of what I know about mirror symmetry and symplectic topology. Comments from Ivan Smith and Nick Sheridan were useful in clarifying the exposition and removing confusing conventions. In addition, I am grateful to Kenji Fukaya for explaining to me Tate's acyclicity theorem and its relevance to mirror symmetry, and to Paul Seidel for having instilled in me the lesson that constructions involving the Fukaya category should first be explained at the cohomological level.

The author was supported by NSF grant DMS-1308179. 
\section{Background}
\label{sec:overview}

\subsection{Lagrangian torus fibrations}
\label{sec:lagr-torus-fibr}

Let $\Q$ be the base of a Lagrangian fibration as in Section \ref{sec:statement-results}. The Arnol'd-Liouville theorem implies that we have canonical identifications
\begin{equation} \label{eq:cotangent_identify}
T_{q} \Q = H^{1}(F_{q}; \bR)  \textrm{ and }  T^{*}_{q} \Q  = H_{1}(F_{q}; \bR) .
\end{equation}
Write $T^{\bZ}_{q} \Q$ for the image of $H^{1}(F_{q}; \bZ) $ in $T_{q} \Q$  under the first isomorphism above, $T^{\bZ} \Q$ for the corresponding rank-$n$ local system on $\Q$, and $T^{*}_{\bZ} \Q$ for the dual. The key property satisfied by this  sublattice of $T^* \Q$ is that its (local) flat sections are closed, hence exact. On a subset $P$ of $\Q$,  a choice of $n$ functions whose differentials span $T^{*}_{\bZ} P$ at every point defines an immersion into $\bR^{n}$ mapping the fibres of $T^{\bZ} P $ to the standard lattice $\bZ^{n}$ in the tangent space of $\bR^{n}$. By choosing $P$ sufficiently small, we obtain a coordinate patch for the \emph{integral affine structure} on $\Q$ induced by $T^{\bZ} \Q $.

The inverse image $X_{P}$ of such a subset under $\pi$ is fibrewise symplectomorphic to $T^{*}P/  T^{*}_{\bZ} P$.  Given sets $P, P' \subset \Q$ which intersect, the restrictions of the two symplectomorphisms  induce a fibrewise symplectomorphism 
\begin{equation}
 T^* \left(P \cap P' \right) / T^{*}_{\bZ} \left(P \cap P' \right) \to  X_{P \cap P'} \to  T^* \left(P \cap P' \right) / T^{*}_{\bZ} \left(P \cap P' \right).
\end{equation}
Such a symplectomorphism can be written as fibrewise addition by a closed $1$-form which is uniquely defined up to an element of $T^{*}_{\bZ} (P \cap P')$.

In order to classify symplectic fibrations which induce a given integral affine structure, choose a finite partially ordered set $\cA$ labelling a simplicial triangulation of $\Q$, i.e. there are vertices labelled by elements of $\cA$, and every cell is the span of a unique ordered subset $I$ of $\cA$.  Assume that this triangulation is sufficiently fine that there is a cover of $\Q$ by codimension $0$ submanifolds with boundary $\{ P_{i} \}_{i \in \cA}$ so that $P_{i}$ contains the open star of the vertex $i$  and all iterated intersections are contractible. Let

\begin{equation}
P_{I} = \bigcap_{i \in I} P_{i}
\end{equation}
and note that $P_{I}$ contains the open star of the cell corresponding to $I$.

Choose a trivialisation $\triv_{i}$ of the inverse image $X_{P_{i}}$, i.e. a fibrewise symplectic identification with $T^{*} P_{i} / T^{*}_{\bZ} P_{i}$.  If $i<j$, the restrictions of $\triv_{i}$ and $\triv_{j}$ to $X_{P_{ij}}$ differ by fibrewise addition of a closed $1$-form; choose a primitive for this $1$-form
\begin{equation}
f_{ij} \co P_{ij} \to \bR.
\end{equation}
If $i < j < k$, the cyclic sum
\begin{equation} \label{eq:representative_alpha}
\alpha_{X} (ijk) = f_{ij} +f_{jk} - f_{ik}
\end{equation}
is function on $P_{ijk}$ whose differential lies in $ T^{*}_{\bZ} P_{ijk}$; i.e. an integral affine function. Such functions define a sheaf on $\Q$ that will be denoted $\Aff$, and $\alpha_{X}$ yields a cocycle in
$ \check{C}^{2}(\Q; \Aff)$.
Write $[\alpha_{X}]$ for the corresponding cohomology class. Here, the \v{C}ech complex associated to a sheaf $\cF$ on $\Q$ is given by
\begin{equation}
\check{C}^{*}_{\cA}(\Q; \cF) = \bigoplus_{\substack{I = (i_0, \ldots, i_{r}) \subset \cA \\ i_0 < \ldots < i_{r}}} \cF(P_{I})[-r]
\end{equation}
with differential given by restriction.

The importance of the sheaf $\Aff$ was noted by Gross and Siebert \cite{gross-siebert-I}, who related it to classical invariants of affine structures. In that spirit, the following is a reformulation of a  result  of Duistermaat \cite[Equation (2.6)]{duistermaat}:
\begin{prop}[c.f. p. 476 of \cite{Clay-book}]
$X$ is determined up to fibrewise symplectomorphism by the triple $(Q, T^{*}_{\bZ} \Q , [\alpha_{X}])$.
\end{prop}
\begin{proof}
The fibrewise symplectic automorphisms of $X_P$ are given by $\Omega^{1}_{c}(P)/T_{\bZ}^{*} P$, where $\Omega^{1}_{c}$ is the sheaf of closed $1$-forms. As noted by Duistermaat, this implies that Lagrangian fibrations which induce the integral affine structure $T_{\bZ}^{*} \Q $ are classified up to fibrewise symplectomorphism by
$
H^1(\Q, \Omega^{1}_{c}/T_{\bZ}^{*} \Q  ).
$

To obtain the desired result, note that the identification of $ \Omega^{1}_{c} $ with $C^{\infty}/\bR$ induces an isomorphism of sheaves 
\begin{equation}
C^{\infty} / \Aff  \cong \Omega^{1}_{c} /T_{\bZ}^{*} \Q.
\end{equation}
Since $C^{\infty}$ is a soft sheaf, this implies the existence of a canonical isomorphism
\begin{equation}
  H^1(\Q, \Omega^{1}_{c}/T_{\bZ}^{*} \Q) \cong H^2(\Q, \Aff).
\end{equation}
\end{proof}

\begin{rem}
The differentiable type of $X$ is determined by a Chern class with values in
$
H^{2}(\Q, T_{\bZ}^{*} \Q)$.
The short exact sequence $
\bR \to \Aff \to T_{\bZ}^{*} \Q $
induces a long exact sequence
\begin{equation}
 \cdots \to H^{2}(\Q, \bR)  \to H^{2}(\Q,\Aff) \to H^{2}(\Q,T_{\bZ}^{*} \Q) \to \cdots
\end{equation}
which has the following interpretation: once a smooth torus fibration over $\Q$ which is compatible with its affine structure is fixed, the set of symplectic structures on the total space for which this fibration is Lagrangian is either empty or an affine space over $ H^{2}(\Q, \bR) $, where the action is given by adding the pullback of a $2$-form on $\Q$.
\end{rem}

\subsection{Construction of the mirror space}
\label{sec:constr-mirr-space}

The next step is to associate a rigid analytic space $Y$  to the integral affine structure $ T^{\bZ} \Q$ on $\Q$. As a set, $Y$ is simply the flat bundle over $\Q$ 
\begin{equation}
Y =  T^{\bZ} \Q \otimes_{\bZ} U_{\Lambda}
\end{equation}
where $ U_{\Lambda}  $ is the multiplicative subgroup of the units in the Novikov ring as in Section \ref{sec:statement-results}. Write
\begin{equation} 
\val \co Y \to \Q
\end{equation}
for the projection. The fibre of $\val$ at a point $q \in \Q$ is
$
H^1(F_{q}; U_{\Lambda})$.

To construct an analytic structure on $Y$, start by considering the valuation
\begin{equation}
H^1(F_{q}; \Lambda^{*}) \to H^1(F_{q}; \bR),
\end{equation}
whose fibre is $ H^1(F_{q}; U_{\Lambda}) $. Splitting the above map by taking a real number $\lambda$ to $T^{\lambda}$,  yields an isomorphism
\begin{equation} \label{eq:split_cohomology_Lambda}
 H^1(F_{q}; U_{\Lambda}) \times H^1(F_{q}; \bR)  \cong H^1(F_{q}; \Lambda^{*}) .
\end{equation}

If $P$ is a sufficiently small neighbourhood of $q \in \Q$, it can be identified using parallel transport with respect to $T^{\bZ} \Q$ with a neighbourhood of the origin in $ T_{q} \Q$. This gives rise to a natural embedding
\begin{equation} \label{eq:decomposition_fibre_over_P}
Y_{P} \equiv \val^{-1}(P)  = \coprod_{p \in P}   H^1(F_{p};  U_{\Lambda}) \subset H^1(F_{q}; \Lambda^{*}).
\end{equation}

 Assume now that $P \subset H^1(F_{q}; \bR)$ is a polytope defined by integral affine equations, i.e. that there exist integral homology classes $\{ \alpha_{i} \} _{i=1}^{d}$, and real numbers $\{ \lambda_i \}_{i=1}^{d}$ such that
 \begin{equation} \label{eq:integral_affine_polygon}
   P = \{ v \in H^1(F_{q}; \bR) |  \langle v , \alpha_{i} \rangle \leq \lambda_i \textrm{ for } 1 \leq i \leq d\}.   
 \end{equation}
 If $P$ is such a polytope, $Y_{P}$ is a \emph{special affine subset} in the sense of Tate \cite[Definiton 7.1]{tate}; these are now usually studied as examples of the more general class of \emph{affinoid domains} \cite{BGR}. The affinoid ring $ \scrO_{P} $ corresponding to $Y_{P}$ in this case consists of formal series
\begin{equation} \label{eq:element_of_reg_function_p}
  \sum_{A \in H_{1}( F_{q}, \bZ)}  f_{A} z^{A}_{q}, \quad f_{A} \in \Lambda
\end{equation}
which $T$-adically converge at every point of $Y_{P}$, i.e. so that
\begin{equation}
 \lim_{|A| \to +\infty} \val(f_{A}) + \langle v,A \rangle= +\infty 
\end{equation} 
whenever $v$ lies in $P$.

\begin{rem}  \label{rem:change_coordinates}
Despite the fact that Equation \eqref{eq:element_of_reg_function_p} refers to the basepoint $q$, the ring $\scrO_{P} $ does not depend on it. One way to see this is to construct a natural isomorphism of the rings associated to different choices of basepoints.  Say that $p \in \Q$ is obtained by exponentiating $v' \in  H^1(F_{q}; \bR) = T_{q}\Q$. Using  parallel transport to identify the tangent space of $q$ with that of $p$, associate to $P$ the polytope
\begin{equation}
P - v \subset   H^1(F_{p}; \bR)  = H^1(F_{q}; \bR).
\end{equation}
Note that the transformation
\begin{equation} \label{eq:map_change_basepoint}
z_{q}^{A} \mapsto t^{\langle v', A \rangle} z_{p}^{A}
\end{equation}
maps series in $z_{q}$ coordinates which are convergent in $P$ bijectively to series in $z_{p}$ coordinates which are convergent in $P - v'$. 
\end{rem}

We now assume that the cover of $\Q$ chosen in Section \ref{sec:lagr-torus-fibr} has the property that 
\begin{equation}
  \label{eq:P-is-polygon}
 \parbox{31em}{for each ordered subset $I \subset \cA$, $P_{I}$ is an integral affine polytope.}
\end{equation}
Covers satisfying this property exist for the following reason: every point in $\Q$ has a neighbourhood which is an integral affine polytope, and two such polytopes intersect along an integral affine polytope whenever they are sufficiently small. Using such a cover, we see that $Y$ is obtained by gluing affinoid sets; it is therefore an affinoid variety.

\subsection{Sheaves as functors}
\label{sec:an-altern-descr}

Consider the category $\scrO_{\cA}$ whose objects are the ordered subsets of $\cA$.  The morphisms in $ \scrO_{\cA} $ are given by
\begin{equation}
  \scrO(I,J) = \begin{cases} \scrO_{J} & \textrm{ if } I \subset  J \\
0 & \textrm{otherwise.}
\end{cases}
\end{equation}
where $\scrO_{I}$ is the ring of functions on $Y_{I}$ (and $Y_{I}$ denotes $Y_{P_{I}}$).  Composition is defined as
\begin{equation}
\scrO(J,K)  \otimes \scrO(I,J) \cong \scrO_{K} \otimes \scrO_{J} \to \scrO_{K} \otimes \scrO_{K} \to \scrO_{K} \cong \scrO(I,K),
\end{equation}
where the middle two arrows are respectively given by restriction (from $Y_{J}$ to $Y_{K}$), and by multiplication (in $\scrO_{K}$).

\begin{defin}
A pre-sheaf of $\scrO$-modules on $Y$ is a functor from $\scrO_{\cA}$ to the category of $\Lambda$-vector spaces.
\end{defin}

To see that this definition is reasonable, recall that a functor $\cF$ assigns to each set $I$ a $\Lambda$-vector space we denote $\scrF(I)$. Since the endomorphisms of the object $I$ in $\scrO_{\cA}$ is the ring of functions on $Y_{I}$, $\scrF(I)$ is equipped with an $\scrO_{I}$ module structure.  Since $\scrO(I,J) = \scrO_{J}$ whenever $I \subset J$, we obtain a map
\begin{equation}
 \scrO_{J} \otimes_{\Lambda} \scrF(I) \to \scrF(J).
\end{equation}
The associativity equation implies that this is a map of $\scrO_{J}$ modules, and that it descends to a map
\begin{equation} \label{eq:restrict_i_j_sheaf}
 \scrO_{J} \otimes_{\scrO_{I}} \scrF(I) \to \scrF(J) ,
\end{equation}
which exactly implies that we have pre-sheaf of $\scrO$-modules  in the usual sense.

\begin{defin}
  A \emph{sheaf of $\scrO$-modules} on $Y$ is a presheaf such that the structure maps in Equation  \eqref{eq:restrict_i_j_sheaf}   are isomorphisms.
\end{defin}
The key point here is that the category of modules over the ring $\scrO_{I}$ is equivalent to the category of sheaves of $\scrO$-modules on the affinoid space $Y_{I}$ (see, e.g. \cite[Section 9.4.3]{BGR}). The datum of a sheaf on $Y_{I}$  can therefore be replaced by that of a single module $\scrF(I)$.  As in the usual description, a sheaf is therefore a presheaf satisfying an additional property.
\begin{rem}
  Since the ring of regular functions on an affinoid domain is Noetherian \cite[p. 222]{BGR}, the notions of coherence and finite generation agree. So we may define a sheaf of coherent modules on $Y$ to be a sheaf of $\scrO$-modules such that each module $\scrF(I)$ is finitely generated; i.e. admits a surjection from a finite rank free module. A standard argument implies that the cohomology modules of finite rank free cochain complexes over $\scrO_{I}$ are coherent modules; it is in this way that coherent sheaves on $Y$ will arise from the mirror.
\end{rem}

\subsection{Rigid analytic gerbes and twisted sheaves}
\label{sec:rigid-analyt-gerb}

There is a natural map
\begin{align}
\exp \co \Aff(P)  & \to \scrO^{*}(Y_P) \\
F & \mapsto t^{F(q)} z_{q}^{dF},
\end{align}
which induces a map
\begin{equation}
H^{2}(\Q, \Aff) \to H^{2}(Y, \scrO^*).
\end{equation}
This map assigns to each Lagrangian fibration over $\Q$ an (analytic) gerbe on $Y$. 

\begin{rem}
The above map is not surjective, but it is reasonably to expect surjectivity by considering deformations of Floer theory in $X$ by the pullback of classes in $H^{2}(\Q, U_{\Lambda})$. For the subgroups $H^{2}(\Q, \bZ_{2})$ and $H^{2}(\Q, 1 + \Lambda_{+})$,  such deformations were considered separately in \cite{FOOO} as \emph{background class} and \emph{bulk deformation}.

\end{rem}

To define a twisted module over this gerbe, one needs a model for sheaf cohomology on $Y$: choose a cocycle $\alpha_{X}$ in $\check{C}^{2}_{\cA}(\Q, \Aff)$ as in Equation \eqref{eq:representative_alpha}; this consists of an assignment $
  \alpha_{X} (ijk) \in \Aff(P_{ijk}) $ for every triple $i < j < k$, satisfying the cocycle condition.  Given a triple $I \subset J \subset K$  of ordered subsets of $\cA$ with final elements $ (i,j,k)$, define
\begin{equation}
 \Aff(P_{K})  \ni \alpha_{X} (IJK) = \begin{cases}
\alpha_{X}(ijk) |P_{K} & \textrm{ if } i < j < k \\
1  & \textrm{otherwise.}
\end{cases}
\end{equation}

Associated to this cocycle, we define a new category $\scrO^{\alpha_{X}}_{\cA}$ with the same objects and morphisms as $\scrO_{\cA}$. The composition is given by
\begin{align}
 \scrO(J,K) \otimes  \scrO(I,J)  & \to \scrO(I,K) \\
f_{K} \otimes f_{J} & \mapsto   \exp(\alpha_{X}(i j k)) \cdot f_{J}|_{Y_{K}} \cdot f_{K}.
\end{align}
The cocycle  property of $\alpha_{X}$  implies that composition is associative. As in the untwisted case, a functor $\scrF$ from $ \scrO^{\alpha_{X}}_{\cA} $ to $\Vect_{\Lambda}$ induces a map of $\scrO_{J}$ modules
\begin{equation} \label{eq:restrict_i_j_twisted_sheaf}
   \scrO_{J} \otimes_{\scrO_{I}} \scrF(I) \to \scrF(J).
\end{equation}

\begin{defin}
  An $\alpha_{X}$-twisted $\scrO$-module is a functor from $\scrO^{\alpha_{X}}_{\cA}$ to $\Lambda$-vector spaces such that the map in Equation \eqref{eq:restrict_i_j_twisted_sheaf} is an isomorphism for every pair $I \subset J$. \qed
\end{defin}

The above definition unwinds into something more familiar: an $\alpha_{X}$-twisted $\scrO$-module over $Y$ is a collection $ \cF(I)$ of $\scrO_{I}$ modules, together with isomorphisms of $\scrO_{J}$ modules
\begin{equation} \label{eq:restriction_module_naive}
\cF_{IJ} \co  \scrO_{J} \otimes_{\scrO_I}   \cF(I)\to \cF(J),
\end{equation}
defined whenever $I \subset  J$, such that
\begin{equation} 
\cF_{JK} \circ    \cF_{IJ}|_{Y_K} =  \exp(\alpha_{X}(IJK))\cdot \cF_{IK}
\end{equation}
for an ordered triple $I \subset J \subset K$. Here, $\cF_{IJ}|_{Y_K}$ denotes the map induced by $\cF_{IJ}$:
\begin{equation} 
\xymatrix{
\scrO_{K} \otimes_{\scrO_I}  \cF(I)  \ar[r]^-{=} &    \scrO_{K} \otimes_{\scrO_{J}} \scrO_{J}  \otimes_{\scrO_I}\cF(I)   \ar[r]^-{\cF_{IJ} } & \scrO_{K} \otimes_{\scrO_I}  \cF(J).}
\end{equation}

\section{Local constructions}
\label{sec:local-constructions}

\subsection{Basics of Floer theory}
\label{sec:basics-floer-theory}
Assume we are given a Lagrangian $L$ so that
\begin{equation} \label{eq:everything_is_good_with_curves}
\parbox{31em}{$L$ is \emph{tautologically unobstructed,} i.e. there exists a tame almost complex structure $J_{L}$ on $X$ so that $L$ bounds no $J_{L}$-holomorphic disc. }
\end{equation}
This is a technical condition, which will allow us to avoid discussing virtual fundamental chains, and should be replaced by the condition that $L$ is unobstructed in the sense of \cite{FOOO}. 

Given a Hamiltonian diffeomorphism $\phi$ mapping $L$ to a Lagrangian transverse to $F_{q}$, there is an ungraded Floer complex 
\begin{equation} 
  CF^*(F_{q},  \phi(L)),
\end{equation}
generated over $\bF_{2}$ by the intersection points of $\phi(L)$ and $F_{q}$. To define the differential,  choose a generic family of almost complex structures $\{ J_ {t} \in \scrJ \}_{t \in [0,1]}$ so that $J_{0} = \phi^{*}(J_{L})$. For each pair $(x,y)$ of intersection points between $\phi(L)$ and $F_{q}$, the space of $J_{t}$-holomorphic maps from the strip $\Strip = \bR \times [0,1]$ to $X$  satisfying the following boundary and asymptotic conditions
\begin{align}
 u(s,1) \in \phi(L)   & \qquad u(s,0) \in F_{q}  \\ 
\lim_{s \to -\infty} u(s,t) = x  & \qquad \lim_{s \to +\infty} u(s,t) = y 
  \end{align}
admits a natural $\bR$-action by translation in the $s$-coordinate. The quotient by this action of the space of such maps is the moduli space of strips $\scrM^{q}(x,y)$,
and the matrix coefficient of $x$ in $d y$ is the count of rigid elements of this moduli space.  The key point here is that this moduli space is regular for generic almost complex structures, so the count of such isolated elements gives a differential by standard methods.

In order for the mirror of $L$ to be an object of the bounded derived category and to be defined away from fields of characteristic $2$, this construction must be refined to a chain complex of free abelian groups which is $\bZ$ graded. Combining the discussions of orientations in \cite{FOOO} and \cite{seidel-Book}, assume that 
\begin{equation}
  w_{2}(L) = \pi^{*}(w_2(\Q)) .
\end{equation}


Under this assumption, one could make an arbitrary choice of $\Pin^{+}$ structure on the bundles 
 $TF_{q} \oplus \pi^{*}(\TQ \otimes |Q|^{\oplus 3})$  and  $TL \oplus \pi^{*}(\TQ \otimes |Q|^{\oplus 3})$ to define the Floer cohomology of $L$ and $F_{q}$.  It will be important to make a \emph{global} choice, i.e. one obtained by restriction from $X$. To this end, identify the restriction of $\pi^{*}(T^*Q)$ to $F_{q}$ with its tangent space via the Arnol'd-Liouville theorem. In particular, a $\Pin^{+}$ structure on
  \begin{equation} \label{eq:Pin_structure_base}
  T^* \Q  \oplus \TQ \oplus |Q|^{\oplus 3} 
  \end{equation}
will induce one by the pullback to all fibres. The above bundle has vanishing second Stiefel-Whitney class, which is the obstruction to such a structure.

Upon fixing $\Pin^{+}$ structures on $TL \oplus \pi^{*}(\TQ \otimes |Q|^{\oplus 3})$  and in Equation \eqref{eq:Pin_structure_base}, index theory assigns a rank $1$ free abelian group $\delta_{x}$ to each intersection point $x \in \phi(L) \cap F_{q}$, with the property that every rigid element of $ \scrM^{q}(x,y) $ induces a canonical map 
\begin{equation} \label{eq:differential_sign}
   d_{u} \co \delta_{y} \to \delta_{x},
\end{equation}
which should be thought of as the signed contribution of $u$ to the differential.

It remains to lift the grading of the Floer complex to a $\bZ$-grading. Equipped with any  almost complex structure for which the fibres are totally real, there is a  natural isomorphism of vector bundles $
  T X \cong \pi^{*}(T \Q) \otimes_{\bR} \bC$. This implies that a density on $\Q$ induces an almost complex quadratic volume form on $X$.  Evaluating such a form on a basis for the tangent space of a Lagrangian defines the phase function
\begin{equation}
\eta_{\Omega} \co  L \to S^{1}.
\end{equation}

By assuming that the phase function on $L$ is null homotopic and fixing a lift to $\bR$, index theory assigns a degree $
\deg(x) \in \bZ $ to every intersection point $x \in \phi(L) \pitchfork  F_{q}$ with the property that the moduli space $\scrM^{q}(x,y) $ has pure dimension
\begin{equation}
  \dim( \scrM^{q}(x,y) ) =  \deg(x) - \deg(y) -1.
\end{equation}
The differential defined in Equation \eqref{eq:differential_sign} then raises degree by $1$ on the Floer complex
\begin{equation} \label{eq:floer_complex_basepoint}
 CF^d( F_{q}, \phi(L)) = \bigoplus_{\deg(x) = d} \delta_{x}.
\end{equation}

\subsection{Convergence of the Floer differential and restriction}
\label{sec:conv-floer-diff}

Let $P$ be a polytope in $\Q$ containing $q$ and $\scrO_{P}$ denote the affinoid ring of $Y_P$. This section adapts an argument of Fukaya showing that, whenever $P$ is sufficiently small, the Floer complex in Equation \eqref{eq:floer_complex_basepoint} is the fibre of a complex of vector bundles on $Y_{P}$
\begin{equation}  \label{eq:free_module_local}
\bigoplus_{x \in F_{q} \cap L} \scrO_{P} \otimes_{\bZ}  \delta_{x} .
\end{equation}
More precisely, choosing $P$ small enough, there is a differential on Equation \eqref{eq:free_module_local} which specialises to the Floer complex
\begin{equation}
CF^*( (F_{p}, b), \phi(L))
\end{equation}
for every point $(p,b) \in Y_{P}$.

In order to define the differential using this moduli space, it is useful to think of the intersection point $ x \in F_{q} \cap \phi(L) $ as a sheet of $\phi(L)$ over $P$. Fixing a trivialisation
\begin{equation}
\triv_{P} \co X_{P} \cong T^{*}P/ T^{*}_{\bZ} P,
\end{equation}
this can be written as the differential of a function $
g_{x} \co P \to \bR $ which is well-defined up to an integral affine function.

\begin{defin} \label{def:conv-floer-diff}
A collection of \emph{Floer data} $D_{P}$ consists of the choices $(\triv_{P}, \phi, J, \{g_{x} \})$. They \emph{are tame in $P$} if there is a (smooth) map $\psi \co P \to \Diff(X)$ which maps $q$ to the identity, such that $\psi_{p}$ maps $F_{q}$ to $F_{p}$  and preserves $\phi(L)$ and the tameness of the almost complex structures $\{ J_{t} \}_{t \in [0,1]}$.
\end{defin}

 Choosing the functions $\{ g_{x} \}$ yields for each strip $u$ with sides mapping to $L$ and $F_{q}$ a class
\begin{equation} \label{eq:homology_class_strip}
[ \partial u] \in H_{1}(F_{q}, \bZ).
\end{equation}
In order to define this class explicitly, note that the choice of trivialisation of $X_{P}$ determines a $0$-section of $X_{P}$, and hence a basepoint on $F_{q}$. The linear path $t dg_{x}$ has endpoints the basepoint at $t=0$ and $x$ at $t=1$. Define $[ \partial u]  $ to be the homology class of the loop obtained by concatenating the paths associated to $dg_{x} $, $dg_{y} $, and the restriction of $u$ to the boundary.

Letting $z^{[\partial u]}$ be the exponential of the unique linear function on $\Q$ which vanishes at $q$ and whose  differential is given by $[ \partial u] $ under the identification of Equation \eqref{eq:cotangent_identify},  define
\begin{equation} \label{eq:differential_local}
d|\delta_{y}= \bigoplus_{x}  \sum_{u \in \scrM^{q}(x,y)} T^{\cE(u)} z^{[\partial u]} \otimes d_{u},
\end{equation}
where $\scrM^{q}(x,y)$ is the moduli space of strips defining the Floer differential, $d_{u}$ is the map induced on determinant lines by $u$, and $ \cE(u)$ is the energy of $u$:
\begin{equation}
  \cE(u) = \int u^{*} \omega.
\end{equation}
We shall presently use an idea of Fukaya to show that the infinite sum in the expression of the differential lies in $\scrO_{P}$.

Recall from Equation \eqref{eq:decomposition_fibre_over_P}  that every element $z' \in Y_{P}$ can be written as a pair $(p,b')$, with $p \in P$ and $b'\in H^1(F_{p}; U_{\Lambda})$.  There is a natural isomorphism of $\Lambda$-modules
\begin{equation} \label{eq:iso_modules}
\bigoplus_{x} \scrO_{P} \underset{z=z'}{\otimes}   \Lambda \otimes_{\bZ}   \delta_{x}  = \bigoplus_{x} \Lambda \otimes_{\bZ}  \delta_{x}  = CF^*((F_{p}, b'), \phi(L)).
\end{equation}

If $q = p$, this map commutes with the differential defined using $J$, which in particular proves that the series in Equation \eqref{eq:differential_local} are convergent at such points. By defining the right hand side of Equation  \eqref{eq:iso_modules} using the almost complex structure $(\psi_{p}^{-1})^{*} J$ (c.f. \cite{Fukaya-cyclic}) so that composition with $\psi_{p}$ gives a bijection between holomorphic strips with boundary on $F_{q}$ and $\phi(L)$, and those with boundary on $F_{p}$ and $\phi(L)$,   convergence is achieved when $q \neq p$. This is where Definition \ref{def:conv-floer-diff} is used. To state the result, define $p-q$ as the point in $T_{q}\Q$ which exponentiates to $p$.

\begin{lem} \label{lem:area_argument_Floer_differential}
If $u$ lies in $ \scrM^{q}(x,y)$, the energy of $ \psi_{p} \circ u$ is
\begin{equation} \label{eq:difference_energy}
     \cE(\psi_{p} \circ u)  =    \cE(u) + \langle p-q, [\partial u] \rangle + g_{x}(q) - g_{y}(q) +g_{y}(p) - g_{x}(p).
\end{equation}
\end{lem}
\begin{proof}

Consider the linear path in $\Q$ from $q$ to $p$. The term $ \langle p-q, [\partial u] \rangle $  is the area of a cylinder in $X$,  lying over this path, and which intersects each fibre in a circle of homotopy class $[\partial u]$. The terms $ g_{x}(p) -  g_{x}(q)  $  and $g_{y}(p) - g_{y}(q)$ are respectively the areas of strips over this path whose intersections with each fibre are the paths from the basepoint to the intersection of each fibre with the local sheets of $\phi(L)$ labelled $x$ and $y$. The right hand side is therefore the sum of the area of $u$ with that of a strip in $X$, which intersects each fibre along the segment from $q$ to $p$ in a path from the intersection with $x$ to the intersection with $y$,  lying in the homotopy class of $ u|\bR \times \{0 \}$.

The result of gluing these two strips is homotopic to $\psi_{p} \circ u $, and Equation \eqref{eq:difference_energy} follows from the invariance of the topological energy under homotopies with fixed Lagrangian boundary conditions.
\end{proof}

By the previous result, the contributions of a curve $u$ to the differentials on the two sides of Equation \eqref{eq:iso_modules} differ by multiplication by
\begin{equation} \label{eq:difference_actions_move_point}
T^{g_{x}(q) - g_{y}(q) +g_{y}(p) - g_{x}(p)}.
\end{equation}
 This readily implies that the pre-composition of the isomorphism in Equation \eqref{eq:iso_modules} with multiplication by
\begin{equation} \label{eq:restriction_map_action}
T^{g_{x}(q) - g_{x}(p)}
\end{equation}
on the $\Lambda \otimes_{\bZ}   \delta_{x}   $ factor is a cochain isomorphism (the disappearance of $\langle p-q, [\partial u] \rangle$ is explained by Remark \ref{rem:change_coordinates}). Gromov compactness applied for all fibres $F_{p}$ over the polygon $P$ implies:
\begin{prop}[c.f. \cite{Fukaya-cyclic}]
  For every pair of intersections $(x,y)$ the series
  \begin{equation}
    \sum_{u \in \scrM^{q}(x,y)} T^{\cE (u)} z^{[\partial u]} 
  \end{equation}
is convergent in $Y_{P}$. \qed
\end{prop}
As a consequence, Equation \eqref{eq:differential_local} defines a differential on the complex 
\begin{equation} 
\cL(Y_P; D_{P}) \equiv \bigoplus_{x \in F_{q} \cap \phi(L)} \scrO_{P} \otimes_{\bZ}  \delta_{x} .
\end{equation}
It will be convenient to drop the Floer data from the notation whenever they are   unambiguously given.

Given an  inclusion of polytopes $ P' \to P$, with basepoints $q$ on $P$ and $q'$ on $P'$, there are restricted data
\begin{equation}
D_{P}|_{P'} \equiv ( \triv_{P}|_{X_{P'}}, \phi,(\psi_{q'}^{-1})^{*}  J, \{g_{x} \}  ) .
\end{equation}
Multiplication of each summand by Equation \eqref{eq:restriction_map_action}  defines a cochain map
\begin{equation}
\cL(Y_P; D_{P})  \to \cL(Y_{P'}; D_{P}|_{P'}).
\end{equation}

\subsection{Change of trivialisation}
Any pair of fibrewise identifications
\begin{align}
\triv_{i} \co X_{P}  & \cong T^*P/ T^{*}_{\bZ} P, \, i \in \{1,2 \}.
\end{align}
differ by the  differential of a function $
f \co P \to \bR$. With respect to two such trivialisations, choose functions $g^{1}_{x}$ and $g^{2}_{x}$ defining every local section of $\phi(L)$, and consider the two sets of data
\begin{align}
  D^{i}  = (\triv_{i}, \phi,  J, \{ g^{i}_{x} \}) , \, i \in \{1,2 \}.
\end{align}

Define the \emph{change of trivialisation} cochain map  
\begin{equation} 
\cL(Y_P; D^{1}) \to \cL(Y_P; D^{2}) 
\end{equation}
as a diagonal map given on the factor $\delta_{x}$ by multiplication with $
  T^{f(q)} z_q^{ d f - dg^{1}_{x} + dg^{2}_{x}} $.

Since this map does not entail counting any holomorphic curves, it is easy to check that given an inclusion $P' \subset P$, we have a commutative diagram 
\begin{equation}
  \xymatrix{ \cL(Y_P; D^{1}) \ar[r] \ar[d]& \cL(Y_P; D^{2}) \ar[d] \\
\cL(Y_{P'};  D^{1}|_{P'} ) \ar[r] & \cL(Y_{P'}; D^{2}|_{P'} ).}
\end{equation}

\subsection{Continuation maps} \label{sec:continuation-maps}

Let $D^{+}$ and $D^{-}$ denote Floer data $ (\triv,  \phi^{\pm}, J^{\pm}, \{g_{x_{\pm}} \}  )$, 
which share a common trivialisation. This section recalls the construction of the continuation map 
\begin{equation} 
  CF^*( F_{q}, \phi^{+}(L)) \to CF^*( F_{q}, \phi^{-}(L))
\end{equation}
as a count of pseudo-holomorphic sections of the projection $
X \times \Strip \to \Strip$.

Pick a path of Hamiltonian diffeomorphisms $\phi^{s}$ such that
\begin{equation} \label{eq:condition_continuation_Hamiltonian}
  \phi^{s} = \begin{cases} \phi^{+} \textrm{ if } 0 \ll s \\
\phi^{-} \textrm{ if } s \ll 0.
\end{cases}
\end{equation}
Recall that there is a unique function $H$ on $X \times \bR$ which generates this flow such that
\begin{equation}\label{eq:normalisation_Hamiltonian}
  \int_{X} H_{s} \omega^{n}  = 0,
\end{equation}
and pick a compactly supported function
\begin{equation}
  G \co X \times B \to \bR,
\end{equation}
which agrees with $H$ on $X \times \bR \times \{1\}$.

If $\alpha$ is a symplectic form on $B$ of finite area, the $2$-form
\begin{equation} \label{eq:form_on_product}
\omega_{X} - dG \wedge ds + C  \alpha
\end{equation}
defines a symplectic structure on $X \times B$ whenever $C$ is a sufficiently large constant. We denote by $\tilde{\scrJ}$ the space of almost complex structures $\tilde{J}$ on $\Strip \times X$ which are of the form
\begin{equation}
\tilde{J} =  \begin{pmatrix}  J & K \\
  0 & j \end{pmatrix} 
\end{equation}
with $J \in \scrJ$, and $K$ vanishes outside a compact set in $\Strip$. Any such almost complex structure will be tamed by the symplectic structure in Equation \eqref{eq:form_on_product} whenever $C$ is sufficiently large. 

We choose an almost complex structure $\tilde{J} \in \tilde{\scrJ}$  whose restrictions to $0 \ll s$ and $s \ll 0$ agree with
\begin{equation} \label{eq:tame_structure_upper_triangular}
\tilde{J}^{\pm} =  \begin{pmatrix}  J^{\pm} & 0  \\
  0 & j \end{pmatrix} .
\end{equation}
\begin{defin} \label{def:elementary_contuation}
  An \emph{elementary continuation datum} $D^{+-}$ from $D^{+}$ to $D^{-}$ is a choice of  the pair of data $(\{\phi^{s}\}, \tilde{J} )$ above.
\end{defin}

For each pair $(x_-, x_+)$ of intersections points, we then define $\scrM^{q}_{\kappa}(x_{-}, x_{+}) $ with respect to the data $D^{+-}$ to be the moduli space of maps $v \co \Strip \to  X$ with $\tilde{J}$-holomorphic graph $  \tilde{v} \co \Strip \to  \Strip \times X$ such that
\begin{align} \label{eq:maps_strips_x_pm}
 v(s,0)  & \in F_{q} \\
v(s,1) & \in \phi^{s} L \\
\lim_{s \to \pm \infty} v(s,t) & = x_{\pm} .
\end{align}

For a generic almost complex structure $\tilde{J}$, these moduli spaces are regular. Computing the linearisation of the index of a solution shows that
\begin{equation}
  \dim( \scrM^{q}_{\kappa}(x_{-}, x_{+}) ) = \deg(x_-) - \deg(x_+).   
\end{equation}
The (topological) energy of a solution to the continuation equation is
\begin{equation}
\cE(v) = \int_{B} v^{*}(\omega) - \int_{\bR}   H_{s}(v(s,1))  ds = \int_{B} \tilde{v}^{*}\left(\omega - dG \wedge  ds \right)
\end{equation}
\begin{lem} \label{lem:bounded_energy_fixed_p}
If $\deg(x_-) = \deg(x_+)$, then for any real number $E$, there are only finitely many elements $v$ of $ \scrM^{q}_{\kappa}(x_{-}, x_{+})  $ such that
\begin{equation}
\cE(v) \leq E.  
\end{equation}
\end{lem}
\begin{proof}
By adding to $\cE$ the total area of the form $C \alpha$ in Equation \eqref{eq:form_on_product}, we obtain the area of $\tilde{v}$ with respect to a symplectic structure on $X \times B$ for which $\tilde{J}$ is tame. A standard application of Gromov compactness therefore implies that the energy  is proper on the Gromov-Floer compactification.

Having excluded bubbling via Assumption \eqref{eq:everything_is_good_with_curves}, the only broken curves in the limit arise when some energy escapes along the ends, which gives rise to components of hypothetical broken curves that are graphs of Floer trajectories. Since the moduli space of Floer trajectories is assumed to be regular, this is impossible whenever the moduli space we are considering has vanishing virtual dimension.
\end{proof}

As in the setting of Floer trajectories, an element $v \in \scrM^{q}_{\kappa}(x_{-}, x_{+})$  induces a canonical map
\begin{equation}
  \kappa_{v} \co \delta_{x_+} \to \delta_{x_-}
\end{equation}
whenever $\deg(x_+) = \deg(x_-)$. Taking the sum over all elements of such rigid moduli spaces defines the continuation map
\begin{align} 
\kappa \co  CF^*(F_{q},  \phi^{+}(L)) & \to CF^*( F_{q}, \phi^{-}(L)) \\
\kappa| \delta_{x_+} & = \sum_{\substack{ \deg(x_+) = \deg(x_-)\\ v \in \scrM^{q}_{\kappa}(x_{-}, x_{+}) }}T^{\cE(v)} \kappa_{v}.
\end{align}

\subsection{Convergence of continuation maps}
\label{sec:conv-cont-maps}

Let $P$ be a polytope based at $q$, with the property that $\cL(Y_P; D^{\pm})$ are well defined. Pick diffeomorphisms $\psi^{\pm}$ as in Definition \ref{def:conv-floer-diff}, and  extend them to a family
\begin{align} \label{eq:diffeo_parametrised_strip}
\Psi \co P \times \Strip & \to \Diff(X) \\  \label{eq:identity_at_basepoint}
\{ q \} \times \Strip & \mapsto \Id
\end{align}
such that the following properties hold for all $p \in P$ (see Figure \ref{fig:conv_continuation}):
\begin{align} 
 \label{eq:diffeo_Lagr_boundary}
\Psi_{p,s,t}  &= \psi^{+}_{p}  \textrm{ if } 0 \ll s \\
\Psi_{p,s,t} & = \psi^{-}_{p} \textrm{ if } s \ll 0 \\ \label{eq:diffeo_maps_fibre_to_fibre}
\Psi_{p,s,0}(F_{q})  & = F_{p}  \\ \label{eq:diffeo_indep_outside_cpct}
\Psi_{p,s,1} & = \Id \textrm{ if }  H_{s} \not \equiv 0.
\end{align}
\begin{figure}
\centering
\begin{tikzpicture}
\begin{scope}                          
\draw[line width=4*\lw] (-4,0)--(4,0);
\draw[line width=4*\lw] (-4,2)--(4,2);

\draw[thick,blue] (-1,3/2)--(-1,2);
\draw[thick,blue] (-1/2,1) arc (-90:-180:1/2);
\draw[thick,blue] (-1/2,1)--(1/2,1);
\draw[thick,blue] (1/2,1) arc (-90:0:1/2);
\draw[thick,blue] (1,3/2)--(1,2);

\draw[thin,blue] (3/2,2)--(3/2,0);
\draw[thin,blue] (-3/2,2)--(-3/2,0);

\node[label=center:{$\psi^{+}_{p}$}] at (3,1) {};
\node[label=center:{$\psi^{-}_{p}$}] at (-3,1) {};
\node at (0,3/2) {$\id$};

\end{scope}
\end{tikzpicture}
\caption{}
\label{fig:conv_continuation}
\end{figure}

For each $p  \in P$, denote by $\tilde{\Psi}_{p}$ the fibrewise diffeomorphism of $\Strip \times X$
\begin{equation}
(s,t,x) \mapsto (s,t,\Psi_{p,s,t}(x)).
\end{equation}
Define $\tilde{J}_{p} = \left(\tilde{\Psi}_{p}^{-1}\right)^{*}\tilde{J}$. This is an almost complex structure on $\Strip \times X$ which has the upper triangular form required in Equation \eqref{eq:tame_structure_upper_triangular}.

\begin{defin} \label{def:conv-cont-maps}
The continuation data $D^{+-}$   \emph{are tame in $P$} if $ \tilde{J}_{p} $   lies in $\tilde{\scrJ}$ for all $p \in P$.
\end{defin}
Of course, tameness depends on the choice of the fibrewise diffeomorphism $\Psi$ of $\Strip \times X$, but as this will be clear from the context, it is omitted. Condition  \eqref{eq:diffeo_indep_outside_cpct} ensures that the off-diagonal term in $ \tilde{J}_{p} $ vanishes outside a compact set.  Openness of the taming condition implies:
\begin{lem}
If  $P$ is sufficiently small the data $D^{+-}$ are tame in $P$. \qed
\end{lem}

Define $ \scrM^{p}_{\kappa}(x_{-}, x_{+}) $ to be the space of maps from a strip to $X$, with boundary conditions $F_{p}$ and $\phi^{s} (L)$, converging to $x_{\pm}$ at the respective ends, whose graphs are $  \tilde{J}_{p}$-holomorphic.
\begin{lem}
Composition with $\tilde{\Psi}_{p}$ defines a bijection 
\begin{equation}
\scrM^{q}_{\kappa}(x_{-}, x_{+}) \cong \scrM^{p}_{\kappa}(x_{-}, x_{+}).
\end{equation}
\end{lem}
\begin{proof}
The key point is that the diffeomorphism $\tilde{\Psi}_{p}$ is compatible with the Lagrangian boundary conditions. In particular, given $v \in \scrM^{q}_{\kappa}(x_{-}, x_{+})  $, Equation \eqref{eq:diffeo_indep_outside_cpct} ensures that the boundary condition of $  \tilde{\Psi}_{p} \circ \tilde{v} $ along $\bR \times \{ 1\}$ is the path $\phi^{s}(L)$, and Equation \eqref{eq:diffeo_maps_fibre_to_fibre} ensures that the boundary condition along $\bR \times \{ 0\}$ is $F_{p}$.
\end{proof}

The advantage of introducing both moduli spaces is that we have independent energy estimates:
\begin{lem} \label{lem:bounded_energy_move_p}
If $\deg(x_-) = \deg(x_+)$, the topological energy defines a proper map on $ \scrM^{p}_{\kappa}(x_{-}, x_{+})  $.
\end{lem}
\begin{proof}
Condition  \eqref{eq:diffeo_indep_outside_cpct} ensures that the off-diagonal term in $ \tilde{J}_{p}$ vanishes outside a compact set, so escape of energy along the ends gives rise to Floer trajectories as in the proof of Lemma \ref{lem:bounded_energy_fixed_p}. The remainder of that proof applies mutatis mutandis.
\end{proof}

These results imply the existence of a cochain map
\begin{equation}
\kappa \co  \cL(Y_P; D^{+}) \to \cL(Y_P; D^{-})
\end{equation}
by counting solutions to continuation maps as follows: recall that the construction of the differential on the two $\scrO_{P}$ modules relied on choosing primitives defining the sheets of $\phi^{+}(L)$ and $\phi^{-}(L)$ over $P$. These primitives define a class $ [\partial v] \in H_{1}(F_{q}, \bZ) $ associated to the boundary of $v \in \scrM^{q}_{\kappa}(x_{-}, x_{+})  $.

The continuation map is defined on each factor by the formula
\begin{equation}
  \kappa | \delta_{x_+} = \bigoplus_{x_-} \sum_{v \in \scrM_{\kappa}(x_{-}, x_{+})  }  T^{\cE(v) } z^{[\partial v]} \otimes \kappa_{v} 
 \end{equation}
The argument proving the convergence of the differential applies verbatim in this case.

\begin{rem}
It is important to note that $\kappa$ is not defined over $\Lambda_0$ since the energy $\cE(v)$ for a solution to the continuation equation is not necessarily positive. This raises obstacles to using the methods developed here to detect quantitative information about Lagrangians in $X$, e.g. displacement energy as in \cite{FOOO-displace}.

\end{rem}

\subsection{Families of continuation maps} \label{sec:famil-cont-maps}
Let $\Delta$ be a compact manifold with boundary parametrising a family $\{ (\{\phi^{s}_{\delta}\}, \tilde{J}_{\delta} ) \}_{\delta \in \Delta}$ of continuation data as in Definition \ref{def:elementary_contuation}.

For each pair $(x_-, x_+)$ of intersections points, let $\scrM^{q}_{\kappa^{\delta}}(x_{-}, x_{+})$ denote the moduli space of continuation maps corresponding to $\delta \in \Delta$, and consider the parametrised space
\begin{equation}
  \scrM^{q}_{\kappa^{\Delta}}(x_{-}, x_{+}) \equiv  \coprod_{\delta \in \Delta}  \scrM^{q}_{\kappa^{\delta}}(x_{-}, x_{+}).
\end{equation}

Assuming the data are chosen generically, this is a manifold with boundary of dimension
\begin{equation}
\deg(x_-) - \deg(x_+)  + \dim(\Delta).
\end{equation}
In particular, if this moduli space is rigid, we may consider the series:
\begin{equation} \label{eq:higher_homotopy_Delta}
\sum_{v \in \scrM^{q}_{\kappa^{\Delta}}(x_{-}, x_{+})} T^{\cE(v)} z^{[\partial v]},
\end{equation}
where the class $[\partial v] \in H_{1}(F_{q}, \bZ)$ is defined as in Equation \eqref{eq:homology_class_strip}.

\begin{prop} \label{prop:famil-cont-maps}
There is a polytope $P_{\Delta} \subset \Q$ so that Equation \eqref{eq:higher_homotopy_Delta} defines an element of $\scrO_{P_{\Delta}}$.
\end{prop}
\begin{proof}
Condition \eqref{eq:condition_continuation_Hamiltonian} and the compactness of $\Delta$ imply that there is a constant $S_{\Delta}$ so that for all $\delta \in \Delta$, $H^{s}_{\delta}$ agrees with $H^{+}$ if $S_{\Delta} \leq s$, and with $H^{-}$ if $s \leq - S_{\Delta}$.  Consider a smooth family $\Psi_{\delta,p,s,t} $ of diffeomorphisms of  $X$ parametrised by $(\delta,p,s,t) \in \Delta \times P \times \Strip$ satisfying, for fixed $\delta \in \Delta$,  Conditions \eqref{eq:identity_at_basepoint}-\eqref{eq:diffeo_indep_outside_cpct}. The assumption that these diffeomorphisms are the identity for $p=q$ and the compactness of $\Delta$ imply that the corresponding fibrewise diffeomorphisms $ \tilde{\Psi}_{\delta,p} $ of $\Strip \times X$ preserve the tameness of the almost complex structure whenever $p$ lies in a sufficiently small neighbourhood of $q$. This yields the analogue of Lemma \ref{lem:bounded_energy_move_p}, and hence convergence in this neighbourhood.
\end{proof}

\subsection{Composition of continuation map}
\label{sec:comp-cont-map}
Let $(\phi^{+}, J^{+})$, $(\phi^{0}, J^{0})$ and $(\phi^{-}, J^{-})$ denote three choices of Hamiltonian diffeomorphisms and almost complex structures, and pick (regular) continuation data $D^{+0}$, $D^{0-}$ and $D^{+-}$ which define cochain maps
\begin{equation} \label{eq:composition_continuation}
\xymatrix{    \cL(Y_P; D^{+}) \ar[rr]^{\kappa_{+-}} \ar[dr]^{\kappa_{+0}}  & & \cL(Y_P; D^{-})  \\
&  \cL(Y_P; D^{0}) , \ar[ur]^{\kappa_{0-}}  &}
\end{equation}
whenever $P$ is sufficiently small. 

Gluing the data $D^{+0}$ and $D^{0-}$, defines a continuation map from $\cL(Y_P; D^{+})  $ to $\cL(Y_P; D^{-})  $  which agrees with the composition $ \kappa_{0-} \circ \kappa_{+0}$. Choosing a homotopy between $D^{+-}$  and the glued data yields a family parametrised by an interval.  Possibly upon shrinking $P$, Equation \eqref{eq:higher_homotopy_Delta} and Proposition \ref{prop:famil-cont-maps} produce a map
\begin{equation} 
  \kappa^{1} \co \cL(Y_P; D^{+}) \to \cL(Y_P; D^{-})[1].
\end{equation}
For each pair of intersections $(x_-, x_+)$ between $\phi^{\pm}(L)$ and $F$ so that $\deg(x_+) = \deg(x_-)$, the moduli space $\scrM^{q}_{\kappa^{1}}(x_{-}, x_{+}) $ had dimension $1$, and its boundary consists of the strata
\begin{align}
&   \scrM^{q}_{\kappa}(x_{-}, x_{+}) \\
&   \coprod_{x_0 \in \phi^{0}(L) \cap F } \scrM^{q}_{\kappa}(x_{-}, x_{0}) \times \scrM^{q}_{\kappa}(x_{0}, x_{+}) \\
&  \coprod_{x'_+ \in \phi^{+}(L) \cap F } \scrM^{q}_{\kappa}(x_{-}, x'_{+}) \times \scrM^{q}(x'_{+}, x_{+})  \\
&  \coprod_{x'_- \in \phi^{-}(L) \cap F }  \scrM^{q}(x_{-}, x'_{-}) \times \scrM^{q}_{\kappa} (x'_{-}, x_{+}).
\end{align}
Counting elements of the first two moduli spaces corresponds to the two compositions in Diagram \eqref{eq:composition_continuation}.  The second two moduli spaces respectively define the composition of $\kappa^{1}$ with the differentials in $ \cL(Y_P; D^{+})$  and $ \cL(Y_P; D^{-})$.

\begin{prop} \label{prop:comp-cont-map}
  If $P$ is sufficiently small, $\kappa^{1}$ defines a homotopy between the two compositions in Diagram \eqref{eq:composition_continuation}. \qed
\end{prop}
Applying the above construction to the null-homotopy for the concatenation of a path and its inverse implies:
\begin{cor} \label{cor:chain_equivalence}
 If $P$ is sufficiently small, $\kappa$ is a chain equivalence. \qed
\end{cor}

\subsection{Compatibility of restriction, continuation, and change of trivilisations}
\label{sec:comp-restr-cont}
Choose data  $(\phi^{\pm}, J^{\pm})$ as in Section \ref{sec:continuation-maps}, and trivalisations $\{ \triv_{i}\}_{i=1,2}$. These yield four sets of Floer data
\begin{equation}
D^{\pm}_{i} \equiv ( \triv_{i}, \phi^{\pm}, J^{\pm}, \{ g^{i}_{x_{\pm}} \} ).
\end{equation}
Using the same continuation equation to map the Floer complexes defined from the data $ D^{+}_{i}$  to those defined from the data $ D^{-}_{i}$, we have:
\begin{lem} \label{lem:two_maps_commute}
 The following diagram, in which the vertical arrows are continuation maps and the horizontal ones are changes of coordinates, commutes
\begin{equation}
  \xymatrix{ \cL(Y_P; D^{+}_{1}) \ar[r] \ar[d]& \cL(Y_P; D^{+}_{2}) \ar[d] \\
\cL(Y_P;  D^{-}_{1} ) \ar[r] & \cL(Y_P; D^{-}_{2} ).}
\end{equation}
\end{lem}
\begin{proof}
 The class in $H_{1}(F_{q}, \bZ) $ associated to a continuation map $v$ depends on the choices of local primitives. We write $[\partial v]^{i}$ for the choice associated to $i$. With this in mind, the commutativity of the diagram reduces to the equality
  \begin{equation}
 [\partial v]^{1} + dg^{1}_{x_+} - dg^{1}_{x_-}  - d f  =  [\partial v]^{2} + dg^{2}_{x_+} - dg^{2}_{x_-}  - d f \in H_{1}(F_{q}, \bZ) \subset T^{*}_{q} P. 
  \end{equation}
\end{proof}

Similarly, given a subpolytope $P' \subset P$ with basepoint $q'$, define $\kappa|P'$ to be the continuation map associated to the data
\begin{equation}
  ( \{ \phi^{s} \},   \tilde{J}_{q'}).
\end{equation}
As in Lemma \ref{lem:two_maps_commute}, there is a commutative diagram 
\begin{equation}
  \xymatrix{ \cL(Y_P; D^{+}) \ar[r]^{\kappa} \ar[d]& \cL(Y_P; D^{-}) \ar[d] \\
\cL(Y_{P'};  D^{+}|_{P'} ) \ar[r]^{\kappa|P'}  & \cL(Y_{P'}; D^{-}|_{P'} ).}
\end{equation}

\section{From Lagrangians to twisted sheaves}
\label{sec:from-lagr-grad}

\subsection{Homological patching}
\label{sec:homological-patching}
In this section,  Floer cohomology groups are used to define an  $\alpha_{X}$-twisted sheaf $H^* \cL$ associated to a Lagrangian $L$.

Start by choosing a simplicial triangulation as in Section \ref{sec:lagr-torus-fibr}, with associated cover $P_{i}$ satisfying Condition \eqref{eq:P-is-polygon}. Denoting by $\scrO_{I}$ the ring of functions on the inverse image $Y_{I}$ of $P_{I}$, this cover should be sufficiently fine that:
\begin{enumerate}
\item for each vertex $i \in \cA$, there are Floer data $D_{i}$ defining complexes of $\scrO_{i}$ modules $ \cL(i) \equiv \cL(Y_{i}; D_{i}) $.
\item for each pair of vertices $i < j$, there are continuation data $D_{ij}$ defining chain equivalences
  \begin{equation}
\cL_{ij} \co   \cL(Y_{ij}; D_{i} | P_{ij})  \to   \cL(Y_{ij}; D_{j}| P_{ij}).
  \end{equation}
\item for each triple of vertices $i < j < k$, there are homotopies $D_{ijk}$ of continuation data between $D_{ik}$ and the gluing of $D_{ij}$ and $D_{jk}$ defining a chain homotopy $\cL_{ijk}$ in the diagram
\begin{equation} 
\xymatrix{    \cL(Y_{ijk}; D_{i}| P_{ijk}) \ar[rr] \ar[dr]  & &  \cL(Y_{ijk}; D_{k}| P_{ijk}) \\
&  \cL(Y_{ijk}; D_{j}| P_{ijk}) . \ar[ur] &}
\end{equation}
\end{enumerate}
 
\begin{lem} \label{lem:fine_triangulation_homological}
If the triangulation of $\Q$ is sufficiently fine, there are choices of data satisfying the above properties.
\end{lem}
\begin{proof}
Start with a finite cover $\scrU_{1}$ by polytopes equipped with tame Floer data. Then choose a second cover $\scrU_{2}$, subordinate to $\scrU_{1}$, so that, for each element of $\scrU_2$ contained in a pair of elements of $\scrU_{1}$, there are convergent continuation data between the two Floer data obtained by restriction, and fix a choice of such data for pairs. Finally, we pick a cover $\scrU_{3}$, subordinate to $\scrU_{2}$, so that there are convergent chain homotopies between all compositions of  the continuation data chosen for $\scrU_{2}$. 

Now, assume that the triangulation of $\Q$ labelled by $\cA$ is subordinate to $\scrU_{3}$ (i.e. so that all open stars of all vertices are contained in an element of the cover), choose the polytope $P_i$ for each element  $i \in \cA$ to also be subordinate to this cover and contain the open star of $i$.  Pick the data $D_i$ arbitrarily among the Floer data associated to elements of $\scrU_{3}$ which contain $P_{i}$. The above choices determine continuation maps  and homotopies satisfying the desired properties.
\end{proof}

\begin{lem}
The modules $H^* \cL(i)  $ and the structure maps $\cL_{ij}$ define an $\alpha_{X}$-twisted sheaf on $Y$.
\end{lem}
\begin{proof}
It suffices to show that the restriction maps commute up to multiplication by $\exp(\alpha_{X}(ijk))$.  Lemma \ref{lem:two_maps_commute} reduces the proof  to the case in which the Floer data $(\phi_{j}, J_{j})$ and $(\phi_{k}, J_{k})$ are obtained by restricting common data $(\phi_{i}, J_{i})$. Let $(f_{ij}, f_{jk},f_{ik} )$  denote the transition functions between the three different trivialisations fixed on $P_{i}$, $P_{j}$ and $P_{k}$. The composition $ \cL_{jk} \circ    \cL_{ij} $ is given on $\delta_{x}$ by multiplication with
\begin{equation}
T^{  f_{jk} + f_{ij}} z^{ d f_{jk} + df_{ij} - dg^{k}_{x} + dg^{i}_{x}} = \exp(f_{jk} + f_{ij}) T^{dg^{i}_{x} - dg^{k}_{x}},
\end{equation}
while the restriction of $ \cL_{ik}  $ to $\delta_{x}$ agrees with
\begin{equation}
T^{  f_{ik}} z^{ d f_{ik} - dg^{k}_{x} + dg^{i}_{x}}= \exp(f_{ik}) T^{dg^{i}_{x} - dg^{k}_{x}}.
\end{equation}
The result now follows immediately from Equation \eqref{eq:representative_alpha}.
\end{proof}

\subsection{Towards the mirror functor}
\label{sec:towards-mirr-funct}
To each ordered subset $I = (i_0, \ldots, i_r) \subset \cA$ corresponding to an $r$-dimensional simplex in $\Q$, Adams' construction \cite{adams} associates an $r-1$-dimensional cube $\sigma_{I}$ of paths in $\Q$ from the initial to the final vertex.  Paths parametrised by the boundary of this cube are given by (i) the family of paths associated to codimension $1$ subsimplices, and (ii) the product of the cubes associated to a pair of complementary simplices in $I$. The homotopy constructed in Section \ref{sec:comp-cont-map} for a triple arose from a family of continuation map associated to such a $1$-dimensional cube in the case $r=2$.

By gluing and induction on dimension, one obtains a family $D_{I}$ of continuation maps from $D_{i_0}$ to $D_{i_r}$ parametrised by $\sigma_{I}$, whose restriction to the boundary strata of $\sigma_{I}$ are given either by the continuation maps associated to a subsimplex, or the concatenation of continuation maps associated to complementary simplices. We are in the setting of Section \ref{sec:famil-cont-maps} so, assuming the parametrised data are chosen generically and the triangulation is sufficiently fine, the count of rigid elements of such a moduli space defines a map
\begin{equation} \label{eq:higher_product_naive}
    \cL_{I} \co   \cL(Y_{i_0}; D_{i_0})  \to   \cL(Y_{I}; D_{i_r}| P_{I})
\end{equation}
of degree $-r$. Adopting the convention that
\begin{equation}
   \cL(I) \equiv \cL(Y_{I}; D_{i_r} | P_{I}),
\end{equation}
the maps in Equation \eqref{eq:higher_product_naive} naturally extend to an $A_{\infty}$ module over $\scrO^{\alpha_{X}}_{\cA} $, i.e an $A_{\infty}$ functor
\begin{equation}
\cL \co \scrO^{\alpha_{X}}_{\cA} \to \Vect_{\Lambda}.
\end{equation}
This data is exactly that of an $\alpha_X$-twisted $A_{\infty}$-presheaf of $\scrO$-complexes on $Y$. Keeping in mind the fact that $\cL(I)$ is a finite rank free $\scrO_{I}$ module, and that the map associated to an inclusion is a quasi-isomorphism (after restriction), $\cL$ in fact defines an object of the $A_{\infty}$-category of $\alpha_X$-twisted sheaves of perfect complexes on $Y$, with respect to the cover $\cA$. In this sense, the assignment $L \to \cL$ gives, at the level of objects, the mirror functor between the derived Fukaya category of $X$ and the derived category of $\alpha_X$-twisted coherent sheaves on $Y$.

\begin{rem}
It is easy in this setting to implement one of the standard equivalences between $A_{\infty}$ and $dg$-modules, and replace $\cL$ by a quasi-equivalent $dg$-module over $\scrO^{\alpha_{X}}_{\cA}$, see e.g. \cite[Theorem 6.15]{fukaya-familyHF}. Since the Fukaya category is an $A_{\infty}$ category, such a replacement does not seem to particularly simplify this approach to Homological mirror symmetry.
\end{rem}

\end{document}